\documentclass[10pt]{amsart}

\usepackage{graphicx}
\usepackage{amssymb,mathrsfs}

\usepackage{enumerate}        

\usepackage[all,dvips]{xy}    
\UseComputerModernTips        

\newtheorem{theorem}{Theorem}[section]
\newtheorem{lemma}[theorem]{Lemma}

\newtheorem{proposition}[theorem]{Proposition}

\theoremstyle{definition}
\newtheorem{definition}[theorem]{Definition}
\newtheorem{example}[theorem]{Example}
\newtheorem{remark}[theorem]{Remark}
\newtheorem{note}[theorem]{Note}

\begin{document}

\title[Hyperplanes in CAT(0) cubical complexes]
{A proof of Sageev's Theorem on hyperplanes in CAT(0) cubical complexes}
\author[D.Farley]{Daniel Farley}
      \address{Department of Mathematics, Miami University \\
               Oxford, OH 45056}
      \email{farleyds@muohio.edu}

\begin{abstract}  
We prove  
that any hyperplane $H$ in a CAT(0) cubical complex $X$ has no
self-intersections and separates
$X$ into two convex complementary components.  These facts were originally
proved by Sageev.  Our argument shows that his theorem is a corollary of
Gromov's link condition.

We also give new arguments establishing some combinatorial properties of hyperplanes.  We show that these
properties are sufficient to prove that the $0$-skeleton of any CAT(0) cubical complex is a discrete median algebra,
a fact that was previously proved by Chepoi, Gerasimov, and Roller.    
\end{abstract}



\maketitle

\section{Introduction}

Two theorems are central in the theory of CAT(0) cubical complexes.  The first is Gromov's well-known link condition.  
A complete statement and proof appear in \cite{BH}.  The second theorem was proved by Sageev in \cite{Sageev}.  
He showed that a group $G$ semisplits over a subgroup
$H$ if and only if $G$ acts on a CAT(0) cubical complex $X$ and there is a hyperplane $J \subseteq X$ such that: i)  the action of $G$
is essential relative to $J$, and ii) the stabilizer of $J$ (as a set) is $H$.  We refer the reader to \cite{Sageev} for details and
definitions.  Sageev's result extends the Bass-Serre theory of groups acting on trees, which says that a group $G$ splits over
$H$ if and only if $G$ acts without inversion on a tree $T$, in which the stabilizer subgroup of some edge $e$ is $H$.  Moreover, 
just as Bass-Serre theory gives a construction of the tree $T$ from the splitting of $G$ over $H$, Sageev gives a construction of 
the CAT(0) cubical complex $X$ from the semisplitting of $G$ over $H$.  Both theories are also alike in that they explicitly describe
the algebraic splittings or semisplittings using their geometric hypotheses.

Both the forward and the reverse directions of Sageev's theorem have significant applications.  The forward direction
(from algebra to geometry) is used in \cite{NR2} and \cite{Wise}, among others.  The proof of the reverse direction uses several properties of hyperplanes
in CAT(0) cubical complexes
(also established in \cite{Sageev}).  Many of these properties are useful in their own right.  For instance, Sageev showed that a hyperplane in a CAT(0) cubical 
complex $X$ has no self-intersections and separates $X$ into two convex complementary components \cite{Sageev}.  
This fact is essential in the proof that groups
acting properly, isometrically, and cellularly on CAT(0) cubical complexes have the Haagerup property \cite{NR1}.  
Sageev establishes the geometric properties of hyperplanes in CAT(0)
cubical complexes using his own system of Reidemeister-style moves.

The main purpose of this note is to offer a new (and, we believe, simpler) 
proof of the following theorem, which we hereafter call ``Sageev's Theorem" for the sake of brevity:
\begin{theorem} \label{thm:intro} \cite{Sageev}
A hyperplane $H$ in a CAT(0) cubical complex $X$ has no self-intersections and separates $X$ into two open convex complementary components.
\end{theorem}
Our proof avoids using Sageev's Reidemeister moves.
The main tool is a block complex $\mathcal{B}(X)$, which is endowed with a natural projection $\pi_{\mathcal{B}}: \mathcal{B}(X) \rightarrow X$.  We apply
a criterion, due to Crisp and Wiest \cite{CW}, for showing that a map between cubical complexes is an isometric embedding.  The criterion is a generalized form
of Gromov's link condition.  We are thus able to 
conclude that the restriction of $\pi_{\mathcal{B}}$ to each connected component of $\mathcal{B}(X)$ is an isometric embedding.
The full statement of Theorem \ref{thm:intro} then follows from the definition of $\mathcal{B}(X)$ after a little more work.

We also give new proofs of some of Sageev's secondary results -- see Subsection \ref{subsec:sageev}, especially Propositions \ref{prop:geodesic} and
\ref{prop:helly}.  Sageev's original proofs used his Reidemeister moves.  Our proofs use techniques from the theory of CAT(0) spaces, including
(especially) projection maps onto closed convex subspaces.  

The paper concludes with some applications.  We sketch a proof of the theorem that every group $G$ acting properly, isometrically, and cellularly on 
a CAT(0) cubical complex has the Haagerup property.  (The first proof appeared in \cite{NR1}.)  We also show that the $0$-skeleton of a CAT(0) cubical complex
is a discrete median algebra under the ``geodesic interval'' operation.  Earlier proofs of the discrete median algebra
property appear in \cite{Chepoi} and \cite{Gerasimov}, and
Martin Roller produced a proof in his unpublished Habilitation Thesis \cite{Roller}.  Our argument is intended to highlight the utility of the combinatorial
lemmas collected in Subsection \ref{subsec:lemmas}, and, in particular, to suggest that the latter lemmas are a sufficient basis
for establishing all of the combinatorial
properties of CAT(0) cubical complexes.  (Indeed, ``discrete median algebra'' and ``CAT(0) cubical complex'' are equivalent ideas, by \cite{Roller}, 
\cite{Gerasimov},
and \cite{Nica}.)  We refer the reader to \cite{CDH} for elegant characterizations of the Haagerup property and property $T$ in terms of median algebras.
          
We note one limitation of the general methods of this paper:  our methods apply only to locally finite-dimensional cubical complexes satisfying Gromov's
link condition.  We need our complexes to be locally finite-dimensional so that their metrics will be complete (see \cite{BH}, Exercise 7.62, page 123).  
In fact, the CAT(0) property
has been established only for locally finite-dimensional cubical complexes satisfying the link condition -- see the passage after Lemma 2.7 
in \cite{Haglund1} for
a useful discussion of this point.  Although our argument is therefore slightly less general than the original one of Sageev, it still covers the
cases that are most commonly encountered in practice.

Section \ref{sec:block} contains a description of the block complex.  Section \ref{sec:lemma} describes the analogue of Gromov's theorem
we need from \cite{CW}.  
Section \ref{sec:main} contains a proof of Sageev's theorem, Theorem \ref{thm:intro}.  Section \ref{sec:combinatorial} collects some essential
combinatorial lemmas.  Finally, Section \ref{sec:applications} contains applications of the main ideas, including proofs that every CAT(0) cubical
complex is a set with walls and that the $0$-skeleton of every CAT(0) cubical complex is a discrete median algebra.

I would like to thank Dan Guralnick for a helpful discussion related to this work, and for telling me about Roller's dissertation.
  
\section{The Block Complex} \label{sec:block}

\begin{definition}
A cubical complex $X$ is \emph{locally finite-dimensional} if the link of each vertex is
a finite-dimensional simplicial complex.
\end{definition}

Throughout the paper, ``CAT(0) cubical complex'' means locally finite-dimensional CAT(0) cubical complex.

\begin{definition}
Let $C \subseteq X$ be a cube of dimension at least one.  A \emph{marking} of $C$ is an equivalence class
of directed edges $e \subseteq C$.  Two such directed edges $e'$, $e''$ are said to be \emph{equivalent}, i.e., to define the same
marking, if there is a sequence of directed edges $e' = e_{0}, \ldots, e_{k} = e''$ such that, for $i \in \{ 0, \ldots, k-1 \}$,
$e_{i}$ and $e_{i+1}$ are opposite sides of a $2$-cell $C_{i} \subseteq C$ and both point in the same direction.  
A \emph{marked cube} is a cube (of dimension at least one) with a marking.
\end{definition}

\begin{example}
Let $X = [0,1]^{3}$, with the usual cubical structure.  We let $C=X$.  There are six markings of $C$.  They are represented by
the directed edges $[(0,0,0),(1,0,0)]$, $[(0,0,0),(0,1,0)]$, $[(0,0,0),(0,0,1)]$, and by the three 
corresponding edges with the opposite directions.
\begin{figure}[!h]
\begin{center}
\includegraphics{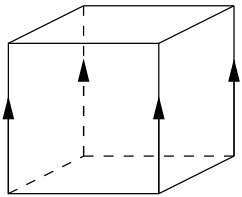}
\end{center}
\caption{The directed edge $[(0,0,0),(0,0,1)]$ determines the marking of the cube.  The $x$-axis is horizontal, and the coordinate system
is a right-handed one.}
\label{fig:cube}
\end{figure}
\end{example}

It is fairly clear from the example that a cube of dimension $n$ has exactly $2n$ markings.  Note that not every face of a marked cube is itself
marked.  In Figure \ref{fig:cube}, the top and bottom faces are unmarked.

\begin{definition}
Let $X$ be a CAT(0) cubical complex.  We let $\mathcal{M}(X)$ denote the space of marked cubes of $X$, which is defined to be
the disjoint union of all marked cubes of $X$.  More formally, $\mathcal{M}(X)$ is the space of triples $(x, C, [e])$, where $C$
is a cube in $X$, $[e]$ is a marking of $C$, and $x \in C$.  For fixed $C$ and $[e]$, the set
$$ C_{[e]} = \{ (x,C,[e]) \mid x \in C \}$$
is an isometric copy of $C$, and $\mathcal{M}(X)$ is the disjoint union of all such sets $C_{[e]}$.  
There is a natural map $\pi_{\mathcal{M}}: \mathcal{M}(X) \rightarrow X$, defined by sending $(x,C,[e])$ to $x$.
\end{definition}

\begin{example}
If $X = [0,1]^{3}$, then $\mathcal{M}(X)$ is a disjoint union of $24$ marked edges, $24$ marked squares, and $6$ marked
three-dimensional cubes.
\end{example}

\begin{definition}
Let $(x_{1}, C_{1}, [e_{1}])$, $(x_{2}, C_{2}, [e_{2}]) \in \mathcal{M}(X)$.
We write $(x_{1}, C_{1}, [e_{1}]) \sim (x_{2}, C_{2}, [e_{2}])$ if:
\begin{enumerate}
\item $x_{1} = x_{2}$, and
\item there is a directed edge $e \in [e_{1}] \cap [e_{2}]$.
\end{enumerate}
\end{definition}

\begin{lemma} \label{lemma:transitive}
The relation $\sim$ is an equivalence relation on $\mathcal{M}(X)$.
\end{lemma}

\begin{proof}
It is already clear that $\sim$ is reflexive and symmetric.

We prove that $\sim$ is transitive.  Thus, we suppose that $(x_{1},C_{1},[e_{1}]) \sim (x_{2},C_{2},[e_{2}])$
and $(x_{2},C_{2},[e_{2}]) \sim (x_{3},C_{3},[e_{3}])$.  Clearly, $x_{1} = x_{2} = x_{3}$.  We can express $C_{2}$
as $C'_{2} \times [0,1]$, where $C'_{2}$ is a cube of dimension one less than the dimension of $C_{2}$, and the
second factor $[0,1]$ is the marked one.  Since $C_{1} \cap C_{2}$ is a marked face of $C_{2}$ (because of the
condition $[e_{1}] \cap [e_{2}] \neq \emptyset$), we must have $C_{1} \cap C_{2} = \widehat{C} \times [0,1]$, for
some non-empty face $\widehat{C} \subseteq C'_{2}$.  Similarly, $C_{2} \cap C_{3} = \widetilde{C} \times [0,1]$, 
for some non-empty face $\widetilde{C} \subseteq C'_{2}$.  Now $C_{1} \cap C_{2} \cap C_{3} \neq \emptyset$, since
$x_{1} \in C_{1} \cap C_{2} \cap C_{3}$.  It follows that $C_{1} \cap C_{2} \cap C_{3} = (\widehat{C} \times \widetilde{C}) \times [0,1]$,
where $\widehat{C} \times \widetilde{C}$ is a non-empty face of $C'_{2}$.

Let us suppose that the marking $[e_{2}]$ of $C_{2}$ is determined by the directed edge $e_{2} = [(v,0), (v,1)]$, where $v$ is a vertex of $C'_{2}$.
It follows easily from the conditions $[e_{1}] \cap [e_{2}] \neq \emptyset$ and $[e_{2}] \cap [e_{3}] \neq \emptyset$ that the directed edge
$[(v',0),(v',1)] \subseteq C_{2}$ is in $[e_{1}]$ (respectively, $[e_{3}]$) if and only if $v' \in \widehat{C}$ (respectively, $\widetilde{C}$).
Thus, if $v$ is a vertex of $\widehat{C} \cap \widetilde{C}$, then $[(v,0),(v,1)] \in [e_{1}] \cap [e_{3}]$.  Such a vertex exists since
$\widehat{C} \cap \widetilde{C} \neq \emptyset$, and this completes the proof.
\end{proof}

\begin{definition}
The \emph{block complex of $X$}, denoted $\mathcal{B}(X)$, is the quotient
$\mathcal{M}(X)/ \sim$.
\end{definition}

\begin{definition} \cite{CW}
A map $f: X \rightarrow Y$ between cubical complexes is called \emph{cubical}
if each cube in $X$ is mapped isometrically onto some cube in $Y$.
\end{definition}

We record the following lemma, the proof of which is straightforward.

\begin{lemma}
The space $\mathcal{B}(X)$ is a cubical complex with a natural cubical map
$\pi_{\mathcal{B}}: \mathcal{B}(X) \rightarrow X$, defined by $\pi(x,C,[e]) = x$. \qed
\end{lemma}

\begin{example}
We describe the cubical complex $\mathcal{B}(X)$ in a special case.  Suppose that
$X = \mathbb{R}^{2}$ with the standard cubulation.  The complex $\mathcal{B}(X)$
consists of an infinite disjoint union of strips having either the form $[m,m+1] \times \mathbb{R}$
or $\mathbb{R} \times [n,n+1]$ ($m,n \in \mathbb{Z}$).  The map 
$\pi_{\mathcal{B}}: \mathcal{B}(X) \rightarrow X$ is ``inclusion''.  Note that there  are two identical
copies of each strip $[m,m+1] \times \mathbb{R}$ in $\mathcal{B}(X)$, since there are two different orientations
for the edge $[m,m+1] \times \{ 0 \}$.  (There are also two copies of $\mathbb{R} \times [n,n+1]$ in $\mathcal{B}(X)$
for a similar reason.)
\end{example}

\section{A geometric lemma} \label{sec:lemma}

The main lemma  of this section (Lemma \ref{lemma:fundamental}) relies heavily on a theorem due to Crisp and
Wiest.

\begin{theorem} (\cite{CW}, Theorem 1(2)) Let $X$ and $Y$ be locally finite-dimensional cubical
complexes and $\Phi: X \rightarrow Y$ a cubical map.  Suppose that $Y$
is locally CAT(0).  The map $\Phi$ is a local isometry if and only if, for
every vertex $x \in X$, the simplicial map 
$Lk(x, X) \rightarrow Lk(\Phi(x), Y)$ induced by $\Phi$ is injective with
image a full subcomplex of $Lk(\Phi(x), Y)$.
\end{theorem}

\begin{proof}
This is exactly Theorem 1(2) from \cite{CW}, except that we allow locally finite-dimensional 
cubical complexes, rather than only finite-dimensional ones.  Since the hypotheses and conclusions
are all local in nature, the proof is unchanged.
\end{proof}

\begin{lemma} \label{lemma:fundamental} Let $X$ and $Y$ be locally finite-dimensional cubical complexes, 
let $Y$ be CAT(0), and assume that $\Phi: X \rightarrow Y$ is a cubical map with
the property that, for every vertex $x \in X$, the simplicial map $Lk(x,X) \rightarrow
Lk(\Phi(x),Y)$ induced by $\Phi$ is injective with image a full subcomplex of $Lk(\Phi(x),Y)$.

For every component $C \subseteq X$, we have:
\begin{enumerate}
\item $C$ is a CAT(0) cubical complex, and
\item $\Phi_{\mid C}$ is an isometric embedding.
\end{enumerate}
\end{lemma}

\begin{proof}
The previous theorem shows that $\Phi$ is a local isometry.  We note that
both $X$ and $Y$ are complete metric spaces, since both are locally finite-dimensional cubical complexes (see Exercise
7.62 on page 123 of \cite{BH}).  Since $Y$ is non-positively curved and $X$
is locally a length space, Proposition 4.14 from page 201 of \cite{BH} applies.  It follows that $X$ is non-positively curved,
the homomorphism $\Phi_{\ast}:\pi_{1}(C) \rightarrow \pi_{1}(Y)$ is injective, and every continuous lifting
$\widetilde{\Phi}: \widetilde{C} \rightarrow \widetilde{Y}$ is an isometric embedding.  Since $\Phi_{\ast}$
is injective, $C$ is simply connected, and therefore $C = \widetilde{C}$, $Y = \widetilde{Y}$, and $\widetilde{\Phi} = \Phi$.
The lemma follows.
\end{proof}

\section{The main theorem} \label{sec:main}

\subsection{A preliminary version of Sageev's theorem}

\begin{theorem}
If $X$ is a locally finite-dimensional cubical complex, then the map
$\pi_{\mathcal{B}}: \mathcal{B}(X) \rightarrow X$ embeds each connected
component of $\mathcal{B}(X)$ isometrically.
\end{theorem}

\begin{proof}
By Lemma \ref{lemma:fundamental}, we need only show that the simplicial map
on links $Lk(v, \mathcal{B}(X)) \rightarrow Lk(\pi_{\mathcal{B}}(v), X)$ is injective
with image a full subcomplex of $Lk(\pi_{\mathcal{B}}(v),X)$.

We choose a vertex $v \in \mathcal{B}(x)$.  Such a vertex can be represented by a vertex
$(x,C,[e])$ in $\mathcal{M}(X)$, where $x \in X^{0}$.  There is a unique directed edge
$e' \in [e]$ containing $x$.  We let $C'$ denote the (undirected) $1$-cell determined by $e'$.
It follows from the definition of $\sim$ that we can represent $v$ by $(x,C',[e'])$.

We let $X_{C'}$ be the subcomplex of $X$ consisting of all closed cubes $C$ such that $C' \subseteq C$.
A marked cube $C_{[e]} \subseteq \mathcal{B}(X)$ touches $v$ if and only if $C' \subseteq C$ and $e' \in [e]$, by the
definition of $\sim$.  Now, for a given cube $C \subseteq X$ such that $C' \subseteq C$, 
there is a unique marking $[e]$ of $C$ such that $e' \in [e]$.  It follows that 
the closed cubes touching $v$ in $\mathcal{B}(X)$ are in one-to-one correspondence with the closed cubes of $X_{C'}$ touching $\pi_{\mathcal{B}}(v)$.
Moreover, given two marked cubes $D_{[e_{1}]}$ and $E_{[e_{2}]}$ such that $e' \in [e_{1}] \cap [e_{2}]$, the intersection
$D_{[e_{1}]} \cap E_{[e_{2}]}$ is mapped isometrically to $D \cap E$ by $\pi_{\mathcal{B}}$, since $D_{[e_{1}]} \cap E_{[e_{2}]} = (D \cap E)_{[e_{3}]}$,
where $[e_{3}]$ is the unique marking of $D \cap E$ determined by the property $[e_{3}] \subseteq [e_{1}] \cap [e_{2}]$. 
It follows that the union of all closed cubes in $\mathcal{B}(X)$ touching
$v$ is combinatorially identical to $X_{C'}$, and the map $\pi_{\mathcal{B}}: \mathcal{B}(X) \rightarrow X$
is locally just the inclusion $X_{C'} \rightarrow X$.  Therefore, the map on links is
injective.

We now consider the image in $Lk(\pi_{\mathcal{B}}(x),X)$.  There is a vertex $v' \in Lk(\pi_{\mathcal{B}}(v),X)$
which is contributed by the $1$-cell $C'$.  The above description of $\pi_{\mathcal{B}}$ implies that
the image of the link $Lk(v, \mathcal{B}(X))$ is the union of all simplices touching $v'$ (i.e., the simplicial
neighborhood of $v'$).  Since $Lk(\pi_{\mathcal{B}}(v),X)$ is flag, this simplicial neighborhood is necessarily
a full subcomplex.
\end{proof}

\subsection{Sageev's theorem}

\begin{definition} \label{def:height}
Fix a component $B$ of the block  complex $\mathcal{B}(X)$.  For each marked cube $C$ of $B$, choose an
isometric characteristic map $c: [0,1]^{n} \rightarrow C$ such that the directed edge $[c(0,0,\ldots,0), c(0,0,\ldots,0,1)]$
represents a marking of $C$.  
If $x \in C$ satisfies $x = c(t_{1}, \ldots, t_{n})$, then the \emph{height} of $x$, denoted $h(x)$, is $t_{n}$.
This height function on marked cubes is easily seen to be compatible overlaps, and induces a height function
$h: B \rightarrow [0,1]$.  We let $B_{t} = h^{-1}(t)$ for $t \in [0,1]$.
\end{definition}

\begin{lemma} \label{lemma:blocks}
\begin{enumerate}
\item For any component $B$ of $\mathcal{B}(X)$ and for any $t \in [0,1]$, $B_{t}$ is a closed convex subset of
$\mathcal{B}(X)$.  The space $\pi_{\mathcal{B}}(B_{t})$ is a closed convex subset of $X$.
\item Each component $B$ of $\mathcal{B}(X)$ factors isometrically as $B_{0} \times [0,1]$.
\item Each $B_{t}$ ($t \in [0,1]$) is a CAT(0) cubical complex.
\end{enumerate}
\end{lemma}

\begin{proof}
\begin{enumerate}
\item It is clear that $B_{t}$ is closed.  

Suppose that $x,y \in B_{t}$.  Let $p: [0, d_{B}(x,y)] \rightarrow B$ be a
path connecting $x$ to $y$.  We can factor each marked cube $C \subseteq B$
of dimension $n$ as $C' \times [0,1]$, where $C'$ is a cube of dimension $n-1$
and the factor $[0,1]$ determines the marking.  There is a natural projection
$\pi_{t}: C \rightarrow C' \times \{ t \}$, and this projection doesn't increase
distances.  Moreover, all such projections are compatible, so in particular
there is a projection $\pi_{t}: B \rightarrow B_{t}$ which fixes $B_{t}$ and
doesn't increase distances.  It follows that $\pi_{t} \circ p$ is a path
in $B_{t}$ which is no longer than $p$.  By the uniqueness of geodesics in CAT(0)
spaces, it follows that any geodesic connecting $x$ to $y$ lies in $B_{t}$.
Therefore, $B_{t}$ is a closed convex subset of $\mathcal{B}(X)$.  Since
$\pi_{\mathcal{B}\mid B}$ is an isometric embedding, $\pi_{\mathcal{B}}(B_{t})$ is
a closed convex subset of $X$.

\item There is a natural map $f: B \rightarrow B_{0} \times [0,1]$, where
$f(x) = (\pi_{0}(x),h(x))$ and $\pi_{0}: B \rightarrow B_{0}$ is the usual
projection onto the closed convex subspace $B_{0}$ (see Proposition 2.4 on page 176 of \cite{BH}).

Assume that $x,y \in B$.  We need to show that
$$ d_{B}(x,y) = \sqrt{ [d_{B_{0}}( \pi_{0}(x), \pi_{0}(y))]^{2} + |h(x) - h(y)|^{2}}.$$
This is clear if $\pi_{0}(x) = \pi_{0}(y)$.  If $\pi_{0}(x) \neq \pi_{0}(y)$, then we consider
the quadrilateral formed by the geodesic segments 
$[\pi_{0}(x), \pi_{0}(y)]$, $[\pi_{0}(x), \pi_{1}(x)]$, $[\pi_{1}(x), \pi_{1}(y)]$, and
$[\pi_{1}(y), \pi_{0}(y)]$.

By Proposition 2.4(3) of \cite{BH}, each of the four resulting Alexandrov angles
measures at least $\pi / 2$.  It therefore follows from the Flat Quadrilateral Theorem
(2.11 from page 181 of \cite{BH}) that all of the angles in the above quadrilateral measure exactly
$\pi /2$, and that the convex hull of $\pi_{0}(x)$, $\pi_{0}(y)$, $\pi_{1}(x)$ and $\pi_{1}(y)$ in $B$
is isometric to a rectangle in Euclidean space.  The desired equality now follows from the definition
of the metric in Euclidean space.
\item It is sufficient to prove this for $B_{0}$.  Since $B = B_{0} \times [0,1]$ is CAT(0),
it must be that each factor is CAT(0) (Exercise 1.16, page 168 of \cite{BH}).  The space $B_{0}$ is a
cubical complex because the identifications in the definition of $B$ are height-preserving.
\end{enumerate}
\end{proof}

\begin{theorem} \label{thm:biggie}
Each hyperplane $\pi_{\mathcal{B}}(B_{t})$ $(0<t<1)$ separates $X$ into two open convex complementary
half-spaces.
\end{theorem}

\begin{proof}
We recall that $\pi_{\mathcal{B}}(B)$ is a closed convex subspace of $X$.  We let 
$\pi: X \rightarrow \pi_{\mathcal{B}}(B)$ be the projection.  By a slight abuse of notation,
we let $h: \pi_{\mathcal{B}}(B) \rightarrow [0,1]$ denote the height function from Definition
\ref{def:height}.

Consider the function $h \circ \pi : X \rightarrow [0,1]$.  We claim
\begin{enumerate}
\item if $[x,y]$ is any geodesic in $X$, then $(h \circ \pi)_{\mid [x,y]}$ must assume
its maximum and minimum values at the endpoints, and
\item if $h(\pi(x)) \in (0,1)$, then $x = \pi(x)$.
\end{enumerate}

We prove (2) first.  Note that, if $h(\pi(x)) \in (0,1)$, then
$\pi(x)$ is an interior point of $\pi_{\mathcal{B}}(B)$.  This is
only possible if $\pi(x) = x$.

We now prove (1).  We assume the contrary.  Assume that $h \circ \pi$ attains its maximum value
on the geodesic $[x,y]$ at neither of the endpoints.  (The case in which $h \circ \pi$ attains
its minimum value at neither of the endpoints is handled in an analogous way.)  We assume
that $h \circ \pi$ attains its maximum value at $z \in [x,y]$, $z \not \in \{ x, y \}$.  It follows
that there is some $t \in (0,1)$ such that 
$$\mathrm{max}\{ (h \circ \pi)(x), (h \circ \pi)(y)\} < t < (h \circ \pi)(z).$$
This implies, by the Intermediate Value Theorem, that there are points $x'$, $y'$
such that $(h \circ \pi)(x') = t = (h \circ \pi)(y')$, where $x'$ lies between $x$ and $z$
on $[x,y]$, and $y'$ lies between $y$ and $z$.  It now follows, from (2), that $x', y' \in B_{t}$.
Since $z \in [x', y'] \subseteq B_{t}$, $(h \circ \pi)(z) = t$, a contradiction.  This proves (1).

We now prove the theorem.  Consider the sets $(h \circ \pi)^{-1}([0,t)) = B_{t}^{-}$ and
$(h \circ \pi)^{-1}((t,1]) = B_{t}^{+}$.  For any $x,y \in B_{t}^{-}$, the geodesic $[x,y]$
is clearly contained in $B_{t}^{-}$ by (1).  It follows that $B_{t}^{-}$ is convex and (therefore)
connected.  By similar reasoning $B_{t}^{+}$ is convex and connected.  Both $B_{t}^{-}$ and
$B_{t}^{+}$ are obviously open, and they are disjoint.  We note finally that
$B_{t}^{-} \cup B_{t}^{+} = X - \pi_{\mathcal{B}}(B_{t})$ (since $\pi_{\mathcal{B}}(B_{t}) =
(h \circ \pi)^{-1}(t)$, by (2)), completing the proof.  
\end{proof}

\begin{definition} \label{def:hyperplane}
A \emph{hyperplane} $H$ in a CAT(0) cubical complex $X$ is the image $\pi_{\mathcal{B}}(B_{1/2})$, where
$B$ is a connected component of $\mathcal{B}(X)$.  We sometimes denote the complementary halfspaces
 $H^{+}$ and $H^{-}$. 
\end{definition}

\begin{note}  
In what follows, we typically identify $\pi_{\mathcal{B}}(B)$ with $B$, and $\pi_{\mathcal{B}}(B_{t})$ with $B_{t}$, 
for the sake of convenience in notation.
\end{note}

\section{Combinatorics of Hyperplanes} \label{sec:combinatorial}

\begin{definition}
Let $X$ be a complete CAT(0) space.  If $C$ is a closed convex subset of $X$, then $\pi_{(X,C)}$ denotes the projection
from $X$ to $C$.  If $x_{1}$, $x_{2}$, and $x_{3}$ are points in $X$, then $\angle^{X}_{x_{2}}(x_{1},x_{3})$ denotes the Alexandrov (or upper)
angle in $X$ between the geodesics $[x_{2}, x_{1}]$ and $[x_{2}, x_{3}]$.  We refer the reader to \cite{BH} for the definitions, which appear on
pages 176 and 9, respectively.
\end{definition}   

\subsection{Three Lemmas} \label{subsec:lemmas}

\begin{lemma} \label{lemma:projection}
Let $H_{1}$, $H_{2}$ be hyperplanes in $X$, and assume that $H_{1} \cap H_{2} \neq \emptyset$.  The
projections $\pi_{(X,H_{i})} : X \rightarrow H_{i}$ and 
$\pi_{(X, H_{i} \cap H_{j})}: X \rightarrow H_{i} \cap H_{j}$ agree on $H_{j}$, where $\{ i, j \} = \{ 1, 2 \}$.
\end{lemma}

\begin{proof}
For the sake of simplicity, we let $j=1$ and $i=2$.  Choose a point $x \in H_{1}$.  We consider the block $B$
containing $H_{1}$, and the projection $\pi_{(B, B \cap H_{2})}: B \rightarrow B \cap H_{2}$.  We denote
the latter projection by $\pi$.  Let $C$ be a marked cube of $B$ containing $\pi(x)$.  We note that $C$ must be at least
two-dimensional, since $C$ meets at least two hyperplanes.
We write $C = C' \times [0,1] \times [0,1]$, where $C' \times \{ 1/2 \} \times [0,1] = H_{2} \cap C$ and
$C' \times [0,1] \times \{ 1/2 \} = H_{1} \cap C$.  

We claim that $\pi(x) \in H_{1}$ (i.e., $\pi(x) \in H_{1} \cap H_{2}$, since $\pi(x) \in H_{2}$ by definition).
Express $\pi(x)$ as $(y, 1/2, t) \in C' \times [0,1] \times [0,1] = C$.  Now since $x \in B_{1/2} = H_{1}$, we have,
by the product decomposition of $B$ (Lemma \ref{lemma:blocks}(2)), 
$$ d(x, \pi(x)) = \sqrt{ D^{2} + |t - 1/2|^{2}},$$
where $D$ is the distance from $x$ to $(y,1/2,1/2)$.  Since $(y,1/2,1/2) \in H_{2} \cap B$
and $\pi(x)$ is the point of $B \cap H_{2}$ closest to $B$, we must have $t = 1/2$.  That is, $\pi(x) = (y,1/2,1/2)$, so
$\pi(x) \in H_{1}$, as claimed.

Next, we claim that $\pi(x) = \pi_{(X,H_{2})}(x)$.  The proof of this fact uses the following characterization of the projection:
if $X$ is a complete CAT(0) space, $C$ is a closed convex subset of $X$, and $x \in X-C$, then $\pi_{(X,C)}(x)$ is the unique element
of $C$ with the property that $\angle^{X}_{\pi_{(X,C)}(x)}(x,z) \geq \pi/2$ 
for all $z \in C - \pi_{(X,C)}(x)$.  We choose $z \in H_{2} - \{ \pi(x) \}$.  Since
$\pi(x)$ is in the interior of $B$, there is some $z' \in [ \pi(x), z]$, $z' \neq \pi(x)$, such that $z' \in B \cap H_{2}$.  By the
definition of $\pi(x) = \pi_{(B, B \cap H_{2})}(x)$, $\angle^{B}_{\pi(x)}(x,z) \geq \pi/2$.  Since $B$ is a convex subset of $X$,
$\angle^{B}_{\pi(x)}(x,z) = \angle^{X}_{\pi(x)}(x,z')$.  It now follows that
$$ \angle^{X}_{\pi(x)}(x,z) = \angle^{X}_{\pi(x)}(x,z') \geq \pi/2,$$
so $\pi(x) = \pi_{(X,H_{2})}(x)$.

Now we argue that $\pi(x) = \pi_{(X, H_{1} \cap H_{2})}(x)$.  If not, then there is $y \in H_{1} \cap H_{2}$ such that
$d_{X}(x,y) < d_{X}(x, \pi(x))$.  This is impossible, however, since $\pi(x)$ is the closest point in $H_{2}$ to $x$.
\end{proof}

\begin{lemma} \label{lemma:intersection}
Assume that $H_{1}$ and $H_{2}$ are hyperplanes, $H_{1} \neq H_{2}$, and $H_{1} \cap H_{2} \neq \emptyset$.  If $e$ is a marked edge perpendicular
to $H_{1}$, then $d_{H_{2}|_{e}}$ is constant.
\end{lemma}

\begin{proof}
Suppose that $e$ is perpendicular to $H_{1}$.  Let $B$ denote the block containing the hyperplane $H_{1}$.
Consider the midpoint of $e$; call it $x$.  We let $\pi$ denote the projection from $X$ onto $H_{2}$.
Let $C$ be a closed marked cube of $B$ which contains $\pi(x)$.  As in the proof of Lemma \ref{lemma:projection},
we write $C = C' \times [0,1] \times [0,1]$, where $H_{1} \cap C = C' \times [0,1] \times \{ 1/2 \}$ and
$H_{2} \cap C = C' \times \{ 1/2 \} \times [0,1]$.  

Since $\pi(x) \in H_{1} \cap H_{2}$ by Lemma \ref{lemma:projection}, one has that $[x, \pi(x)] \subseteq B_{1/2} = B_{0} \times \{ 1/2 \}$.
We can express $[x, \pi(x)]$ as $[ \pi_{0}(x), \pi_{0}(\pi(x))] \times \{ 1/2 \}$, where $\pi_{0}$ denotes the projection
from $B$ to $B_{0}$.  If $y$ is some other point on $e$, then $[\pi_{0}(x), \pi_{0}(\pi(x))] \times \{ h(y) \}$ is a geodesic connecting
$y$ to a point in $H_{2}$.  It follows that $d_{H_{2}}(y) \leq d_{H_{2}}(x)$, for all $y \in e$.  

One argues that equality always holds
by the convexity of the function $d_{H_{2}}$ (see Corollary 2.5 on page 178 of \cite{BH}).  Indeed, suppose that $y_{1}$, $y_{2} \in e$, where
$h(y_{1}) < h(x) < h(y_{2})$, and $d_{H_{2}} (y_{i}) < d_{H_{2}}(x)$ for at least one index $i \in \{ 1, 2 \}$.  The function $d_{H_{2}}$ is
concave up (i.e., convex) and non-constant on the geodesic $[y_{1}, y_{2}]$, and attains a maximum value of $d_{H_{2}}(x)$ at the interior point $x$.
This is a contradiction.      
\end{proof}

\begin{lemma} (\cite{F1}, Lemma 2.6(4)) \label{lemma:4pt}
If $H_{1}$ and $H_{2}$ are hyperplanes, $H_{1}^{+} \cap H_{2}^{+}$, $H_{1}^{-} \cap H_{2}^{+}$, $H_{1}^{+} \cap H_{2}^{-}$, and $H_{1}^{-} \cap H_{2}^{-}$ are
all non-empty, then $H_{1} \cap H_{2} \neq \emptyset$.
\end{lemma}

\begin{proof}
Assume that the four intersections in the hypothesis are all non-empty
and $H_{1} \cap H_{2} = \emptyset$.  It follows that $\{ H_{1}^{+} \cup H_{2}^{+},
H_{1}^{-} \cup H_{2}^{-} \}$ is an open cover of $X$.  Each of the half-spaces
$H_{1}^{+}$, $H_{1}^{-}$, $H_{2}^{+}$, and $H_{2}^{-}$ is a convex subspace of
the CAT(0) space $X$, and therefore contractible.  Each of the four intersections
in the hypothesis is contractible for the same reason.

It follows that the sets $X^{+} = H_{1}^{+} \cup H_{2}^{+}$ and $X^{-} = H_{1}^{-} \cup H_{2}^{-}$
are simply connected, since each is the union of two open contractible sets which intersect in an open
contractible set.  The intersection $X^{+} \cap X^{-}$ is the disjoint union of two open contractible sets:
$H_{1}^{+} \cap H_{2}^{-}$ and $H_{2}^{+} \cap H_{1}^{-}$.  Let $c$ be an arc contained in $X^{+}$, connecting
$H_{1}^{+} \cap H_{2}^{-}$ to $H_{2}^{+} \cap H_{1}^{-}$, and meeting each in an open segment.

We apply van Kampen's theorem to the pieces $X^{-} \cup c$ and $X^{+}$.  The first piece $X^{-} \cup c$ satisfies
$\pi_{1}( X^{-} \cup c ) \cong \mathbb{Z}$, while the second is simply connected.  The intersection of these two pieces
is the simply connected set $(H_{1}^{+} \cap H_{2}^{-}) \cup (H_{2}^{+} \cap H_{1}^{-}) \cup c$.  It follows that
$\pi_{1} ( X^{-} \cup X^{+}) = \pi_{1}(X)$ is isomorphic to $\mathbb{Z}$.  Since $X$ is CAT(0), it must be contractible.
This is a contradiction.
\end{proof}

\subsection{Sageev's Combinatorial Results} \label{subsec:sageev}
We cover only some basic combinatorial results in this subsection.  A more
advanced treatment of the combinatorial properties of hyperplanes appears
in an appendix to \cite{HW}.

\begin{proposition} \label{prop:geodesic}
\cite{Sageev}  An edge-path $p$ in $X^{1}$ is geodesic if and only if $p$ crosses any given hyperplane $H$
at most once.
\end{proposition}

\begin{proof}
We first prove the forward direction.
Suppose, on the contrary, that a certain geodesic edge-path crosses some hyperplane more than once.  We consider
a shortest geodesic edge-path $p$ which crosses some hyperplane multiple times.  We write $p = (e_{1}, \ldots, e_{n})$,
and let $H_{1}, \ldots, H_{n}$ denote the hyperplanes crossed by the edges $e_{1}, \ldots, e_{n}$ (respectively).  Since $p$
is the shortest edge-path with the given property, we must have $H_{1} = H_{n}$, but there are no other repetitions in the list
$H_{1}, \ldots, H_{n}$ (i.e., a total of $n-1$ distinct hyperplanes are crossed by $p$).  We let $H_{1}^{-}$ denote the component of $X-H_{1}$
containing $\iota(e_{1})$ and $\tau(e_{n})$.  Clearly the other component of $X-H_{1}$, denoted $H_{1}^{+}$, contains the edge-path
$(e_{2}, \ldots, e_{n-1})$.  We adopt the convention that $\iota(e_{j}) \in H_{j}^{-}$ and $\tau(e_{j}) \in H_{j}^{+}$, for
$j \in \{ 2, \ldots, n-1 \}$.  

Consider an edge $e_{j}$, $j \in \{ 2, \ldots, n-1 \}$.  Note that $\iota(e_{1}) \in H_{1}^{-} \cap H_{j}^{-}$,
$\iota(e_{j}) \in H_{1}^{+} \cap H_{j}^{-}$, $\tau(e_{j}) \in H_{1}^{+} \cap H_{j}^{+}$,
and $\tau(e_{n}) \in H_{1}^{-} \cap H_{j}^{+}$.  It follows that the hyperplanes $H_{1}$ and $H_{j}$ intersect, for $j \in \{ 2, \ldots, n-1 \}$,
by Lemma \ref{lemma:4pt}.

We now apply Lemma \ref{lemma:intersection}.  
Since $d( \iota(e_{2}), H_{1}) = 1/2$ and $d_{H_{1}}$ is constant on $e_{2}$, we must have $d_{H_{2}}(x) = 1/2$ for all
$x$ in $e_{2}$.  We can inductively conclude that $d_{H_{1}}(x) = 1/2$ for all $x$ in $(e_{2}, \ldots, e_{n-1})$.

It follows that the entire edge-path $p = (e_{1}, \ldots, e_{n})$ is contained in the block $B$ containing $H_{1}$.  The edges
$e_{2}, \ldots, e_{n-1}$ are all unmarked edges in the block $B = B_{0} \times [0,1]$.  We identify $\iota(e_{2})$ with a vertex
$(v',1) \in B$ and $\tau(e_{n-1})$ with a vertex $(v'',1) \in B$.  It follows that $\iota(e_{1}) = (v', 0)$ and $\tau(e_{n}) = (v'',0)$.
The edge-path $(e_{2}, \ldots, e_{n-1})$ connects $(v',1)$ to $(v'',1)$.  There is a corresponding edge-path $(e'_{2}, \ldots, e'_{n-1})$ 
connecting $(v',0)$ to $(v'',0)$.  This contradicts the fact that $p$ is geodesic.

Now suppose that $p$ crosses any given hyperplane $H$ at most once.  It follows that the endpoints $\iota(p), \tau(p)$ 
of $p$ are separated
by all of the hyperplanes crossed by $p$.  If we assume that there are $n$ such hyperplanes in all (and so $p$ has length $n$),
then any edge-path $q$ from $\iota(p)$ to $\tau(p)$ must cross all $n$ of these hyperplanes, so the length of $q$ is at least $n$.  
It follows that $p$ is geodesic.  
\end{proof}

\begin{definition} Suppose that $(e_{1},e_{2})$ is an edge-path in a CAT(0) cubical complex $X$ such that $e_{1}$ and $e_{2}$ are
perpendicular sides of a square $C$ in $X$.  We let $e'_{i}$ denote the side of $C$ opposite $e_{i}$, for $i = 1,2$.  The operation
of replacing $(e_{1},e_{2})$ by the edge-path $(e'_{2},e'_{1})$ is called a \emph{corner move}. Note that the edge-paths $(e_{1}, e_{2})$
and $(e'_{2}, e'_{1})$ have the same endpoints. 
\end{definition}

\begin{proposition} \label{prop:corner}
If $(e_{1},e_{2})$ is an edge-path in $X$, $e_{i}$ crosses the hyperplane $H_{i}$ ($i=1,2$), $H_{1} \neq H_{2}$, 
and $H_{1} \cap H_{2} \neq \emptyset$, then the edges $e_{1}$ and $e_{2}$ are perpendicular sides of a square $C$ in $X$.
\end{proposition}
   
\begin{proof}
Let $B$ denote the block containing the hyperplane $H_{1}$.  We write $B = B_{0} \times [0,1]$, and assume that $\iota(e_{1}) = (v, 0)$,
for some vertex $v \in B_{0}$.  It follows that $\tau(e_{1}) = (v,1)$.  Since $H_{2} \cap H_{1} \neq \emptyset$ and $H_{1} \neq H_{2}$, 
we have that $d_{H_{1}}$ is constant on $e_{2}$, by Lemma \ref{lemma:intersection}.  In particular, $d_{H_{1}} (x) = 1/2$, for any $x$ on 
the edge $e_{2}$, since $d(\iota(e_{2}), H_{1}) = 1/2$.  It follows that $e_{2}$ has the form $[ (v,1), (v',1) ]$, where $[v,v']$ is an
edge in $B_{0}$.  Therefore the edge-path $(e_{1}, e_{2})$ forms one half of the boundary of the square $(v,v') \times [0,1] \subseteq B$,
as desired.
\end{proof}  

\begin{proposition} \label{prop:helly}
\cite{Sageev}  If $H_{1}, \ldots, H_{n}$ satisfy $H_{i} \cap H_{j} \neq \emptyset$ for any $i,j \in \{ 1, \ldots, n \}$, then
$H_{1} \cap \ldots \cap H_{n} \neq \emptyset$.
\end{proposition}

\begin{proof}
The proof is by induction on $n$.  The conclusion is obvious if $n=2$.  We suppose that $n>2$.  By induction, $H_{1} \cap \ldots \cap H_{n-1} \neq \emptyset$,
so we take $x \in H_{1} \cap \ldots \cap H_{n-1}$.  By Lemma \ref{lemma:projection}, 
$\pi_{(X,H_{n})}(x) = \pi_{(X,H_{j} \cap H_{n})}(x)$ for $j \in \{ 1, \ldots, n-1 \}$.  It follows
that $\pi_{(X,H_{n})}(x) \in H_{1} \cap \ldots \cap H_{n}$.
\end{proof}

\section{Applications} \label{sec:applications}

\subsection{The set-with-walls property}

\begin{definition} (first defined in \cite{HP})
Let $S$ be a set.  A \emph{wall} $W$ in $S$ is a partition $\{ W^{-}, W^{+} \}$ of $S$.
Two points $x,y \in S$ are \emph{separated} by the wall $W$ if $x \in W^{-}$ and $y \in W^{+}$
(or vice versa).  We say that $(S, \mathcal{W})$ is a \emph{set with walls} if $\mathcal{W}$ is a collection
of walls in $S$ such that, for any $x,y \in S$, at most finitely many walls $W \in \mathcal{W}$ separate
$x$ from $y$.

If $G$ is a group and $S$ is a $G$-set, then $(S, \mathcal{W})$ is a \emph{$G$-set with walls} if the natural
action of $G$ permutes the set $\mathcal{W}$.
\end{definition}

\begin{definition}
If $(S, \mathcal{W})$ is a set with walls, then the \emph{wall pseudometric} $d_{(S,\mathcal{W})}: S \times S \rightarrow \mathbb{R}^{+}$
is defined by 
$$ d_{(S,\mathcal{W})}(x,y) = | \{ W \in \mathcal{W} \mid W ~separates~ x ~from~ y \}|.$$
If $(S, \mathcal{W})$ is a $G$-set with walls, then we say that $G$ acts \emph{properly} on $(S, \mathcal{W})$ if, for any $r > 0$ and $x \in S$,
the set
$$ \{ g \in G \mid d_{(S,\mathcal{W})}(x,gx) < r \}$$
is finite.
\end{definition}

\begin{remark}
It is straightforward to check that $d_{(S,\mathcal{W})}$ is symmetric and satisfies the triangle inequality, and that $G$ acts isometrically
on  $(S, \mathcal{W})$ if the latter is a $G$-set with walls.
\end{remark}

\begin{theorem}
If $X$ is a CAT(0) cubical complex, then $(X^{0}, \mathcal{W})$ is a set with walls, where $\mathcal{W} = \{ \{ H^{+} \cap X^{0}, H^{-} \cap X^{0} \} \mid
$H$ ~is~a~hyperplane~in~ X \}$.  If $G$ acts cellularly and by isometries on $X$, then $(X^{0}, \mathcal{W})$ is a $G$-set with walls.  If $G$ acts properly
on $X$, then $G$ acts properly on $(X^{0}, \mathcal{W})$.
\end{theorem}

\begin{proof} (Sketch)
The fact that $\{ H^{+} \cap X^{0}, H^{-} \cap X^{0} \}$ is a wall follows from Theorem \ref{thm:biggie}; 
the fact that two vertices $x$, $y$ are separated by at most finitely many
walls $W_{H} = \{ H^{+} \cap X^{0}, H^{-} \cap X^{0} \} \in \mathcal{W}$ follows from the fact that a wall $W_{H}$ separates $x$ from $y$ if and only if
a geodesic edge-path from $x$ to $y$ crosses $H$.  The remaining statements are similarly straightforward to check.
\end{proof}

We note that \cite{CN} contains a proof of the converse:  there is a construction of a CAT(0) cubical complex associated to any space with walls.

\begin{definition}
A discrete group $G$ has the \emph{Haagerup property} if there is a proper affine isometric action of $G$ on a Hilbert space
$\mathcal{H}$.  Here ``proper'' means
metrically proper:  if $v \in \mathcal{H}$ and $r > 0$ are given, then $|\{ g \in G \mid d(v, g \cdot v) < r \}| < \infty$.
\end{definition}

\begin{theorem} \cite{NR1}
If $G$ acts properly, cellularly, and by isometries on a CAT(0) cubical complex $X$, then $G$ has the Haagerup property.
\end{theorem}

\begin{proof}
(Sketch)  One chooses a basepoint $v \in X^{0}$ and orientations for all hyperplanes $H \subseteq X$.  Let $\mathcal{W}^{or}$
denote the set of oriented hyperplanes.  The group $G$ acts as (infinite) signed permutation matrices on the Hilbert space
$\ell^{2}(\mathcal{W}^{or})$.  For $g \in G$, we let 
$$\delta(g) = \sum \pm H,$$
where the sum is over all hyperplanes separating $v$ from $gv$.  Here $H$ is counted with the plus sign if one crosses $H$ in the direction
of its given orientation when moving from $v$ to $gv$, and it is counted with the minus sign otherwise.

The action $\alpha: G \times \ell^{2}(\mathcal{W}^{or}) \rightarrow \ell^{2}(\mathcal{W}^{or})$ given by $\alpha(g,v) = g \cdot v + \delta(g)$
has the desired properties.
\end{proof}

\subsection{The median algebra property}

Let $\mathcal{P}(S)$ denote the power set of $S$.

\begin{definition}
A \emph{median algebra} is a set $S$ together with an interval operation
$[,]: S \times S \rightarrow \mathcal{P}(S)$ such that
\begin{enumerate}
\item $[x,x] = \{ x \}$ for $x \in S$;
\item $[x,y] = [y,x]$ for $x,y \in S$;
\item If $z \in [x,y]$, then $[x,z] \subseteq [x,y]$;
\item For any $x,y,z \in S$,
$[x,y] \cap [y,z] \cap [x,z]$ is a singleton set.  The unique element
of this singleton set, denoted $m(x,y,z)$, is called the \emph{median} of
$x$, $y$, $z$.
\end{enumerate}
A median algebra is \emph{discrete} if each set $[x,y]$ is finite.
\end{definition}

\begin{definition}
Assume that $X$ is a CAT(0) cubical complex.  If $x, y \in X^{0}$, then 
the \emph{geodesic interval} $[x,y]$ is the set of all vertices $z \in X^{0}$
that lie on some geodesic edge-path connecting $x$ to $y$.
\end{definition}

\begin{remark}
Note, for instance, that the geodesic interval between two integral points $(a,b)$ and $(c,d)$  ($a \leq c$ and $b \leq d$) 
in $\mathbb{R}^{2}$ is $\{ (x,y) \mid a \leq x \leq c; \, \, b \leq y \leq d; \, \, x,y \in \mathbb{Z} \}$.
\end{remark}

\begin{theorem}
Let $X$ be a CAT(0) cubical complex.  The set of vertices $X^{0}$ is a discrete
median algebra, where the interval operation $[,]: X^{0} \times X^{0} \rightarrow \mathcal{P}(X^{0})$
is the geodesic interval.
\end{theorem}

\begin{proof}
Properties (1) and (2) are clear.

We now prove (3).  Let $z \in [x,y]$.  This means that there is a geodesic edge-path $p$ connecting $x$ to $y$ and
passing through $z$.  We can express $p$ as $p_{1} \cup p_{2}$, where $p_{1}$ is a geodesic edge-path connecting $x$ to $z$
and $p_{2}$ is a geodesic edge-path connecting $z$ to $y$.  If $w \in [x,z]$, then there is a geodesic edge-path $p'_{1}$
connecting $x$ to $z$ and passing through $w$.  Since $p'_{1}$ and $p_{1}$ have the same length, $p'_{1} \cup p_{2}$ is
also a geodesic edge-path connecting $x$ to $z$, and it passes through $w$.  Therefore
$w \in [x,y]$.  It follows that $[x,z] \subseteq [x,y]$, proving (3).

We now prove that $[x,y]$ is always finite.  If $H_{1}, \ldots, H_{n}$ are the hyperplanes separating $x$ from $y$, then,
by Proposition \ref{prop:geodesic}, an edge-path
$p$ is a geodesic edge-path connecting $x$ to $y$ if and only if $p$ begins at $x$ and crosses exactly the hyperplanes $H_{1}, \ldots, H_{n}$.
However, such an edge-path is uniquely determined by the order in which the hyperplanes $H_{1}, \ldots, H_{n}$ are crossed.  It
follows that there are at most $n!$ geodesic edge-paths, each of which passes through only finitely many points, so $|[x,y]| < \infty$.

We now prove (4).  Fix $x,y,z \in X^{0}$.  We first show that $[x,y] \cap [y,z] \cap [x,z]$ contains at most one element.  Suppose 
$v, w \in [x,y] \cap [y,z] \cap [x,z]$ and $v \neq w$.  There is a hyperplane $H$ separating $v$ from $w$.  It must be that two (or more)
elements of $\{ x, y , z \}$ lie in one of the complementary components of $X-H$.  It follows without loss of generality (i.e., up to relabelling)
that $v$ is separated from both $x$ and $y$ by $H$.  Since $v \in [x,y]$ by our assumption, there is a geodesic edge-path $p$ from $x$ to $y$
passing through $v$.  The geodesic edge-path $p$ would necessarily cross $H$ twice, however.  This is a contradiction.

We now need to show that $[x,y] \cap [y,z] \cap [x,z]$ is non-empty.  We do this by induction on $d(x,y) + d(y,z) + d(x,z)$, where $d$
denotes the edge-path (or combinatorial) distance.  The base case is trivial.  For
the inductive step, we need a definition.  If a hyperplane $H$ separates both $x$ and $y$ from $z$, then we say that $H$ is an \emph{$\{ x, y  \}$-hyperplane}.
We can similarly define $\{ x,z \}$- and $\{ y, z \}$-hyperplanes.  Note that any hyperplane crossed by an edge-path geodesic between any two points
of $\{x,y,z \}$ must be a $\{ a, b \}$-hyperplane, where $\{ a, b \} \subseteq \{ x, y, z \}$.  
If $z \in [x,y]$, $x \in [y,z]$, or $y \in [x,z]$, then the desired conclusion
is clear,  so we assume that none of $x$, $y$, and $z$ is contained in the interval of the other two.  We choose geodesic edge-paths $p_{x}$, $p_{y}$ connecting
$z$ to $x$ and $y$, respectively.

We claim that there is some $\{ x,y \}$-hyperplane $H$ that is crossed by both $p_{x}$ and $p_{y}$.  Indeed, $p_{x}$ crosses only
$\{ x,y \}$- and $\{ y, z \}$-hyperplanes by definition, and $p_{y}$ crosses only $\{ x,y \}$- and $\{ x,z \}$-hyperplanes.  Thus,
if no $\{ x,y \}$-hyperplane is crossed by both $p_{x}$ and $p_{y}$, then $p_{x}^{-1}p_{y}$ crosses no hyperplane more than once,
and is therefore geodesic.  Since $p_{x}^{-1}p_{y}$ passes through $z$, we have $z \in [x,y]$, a contradiction.

Next, we claim that there are geodesic edge-paths $p'_{x}$ and $p'_{y}$ from $z$ to $x$ and $y$ with the property that $p'_{x}$ and $p'_{y}$ cross all
$\{ x,y \}$-hyperplanes before crossing any $\{ x,z \}$- or $\{ y,z \}$-hyperplanes.  We prove only that there is such a $p'_{x}$, since the proof that
there is such a $p'_{y}$ is similar.  To establish the existence of the desired $p'_{x}$, it is sufficient to show that, whenever $p_{x}$ crosses
a $\{ y,z \}$-hyperplane $H'$ before an $\{ x,y \}$-hyperplane $H''$, $H' \cap H'' \neq \emptyset$, for then we can use corner moves to change $p_{x}$
into the desired $p'_{x}$.  We assume the convention that $z \in (H')^{-} \cap (H'')^{-}$.  If $e'$ is the (unique) edge of $p_{x}$ crossing $H'$, then
$\tau(e') \in (H')^{+} \cap (H'')^{-}$.  If $e''$ is the edge of $p_{x}$ crossing $H''$, then $\tau(e'') \in (H')^{+} \cap (H'')^{+}$.
Now note that $y \in (H')^{-} \cap (H'')^{+}$.  We now have $H' \cap H'' \neq \emptyset$, by Lemma \ref{lemma:4pt}.  This proves the claim.

We therefore have $p'_{x}$ and $p'_{y}$ (as above).  Let $H_{1}$ be the first hyperplane crossed by $p'_{x}$.  It is, of course,
an $\{ x,y \}$-hyperplane.  We claim that we can alter $p'_{y}$ to obtain a new geodesic edge-path $p''_{y}$ connecting $z$ to $y$,
such that $p''_{y}$ crosses $H_{1}$ first.  (We note that $p'_{y}$ must cross $H_{1}$, since $H_{1}$ separates $z$ from $y$ by definition.)
It is enough to show that if the hyperplane $\{ x,y \}$-hyperplane $H_{2}$ is crossed by $p'_{y}$ before $H_{1}$, then $H_{1} \cap H_{2} \neq \emptyset$,
for then we can alter $p'_{y}$ by corner moves in order to arrive at the desired $p''_{y}$.  We assume the convention that $z \in (H_{1})^{-} \cap (H_{2})^{-}$.
If $e_{2}$ is the edge of $p'_{y}$ crossing $H_{2}$, then $\tau(e_{2}) \in (H_{1})^{-} \cap (H_{2})^{+}$.  If $e_{1}$ is the edge of $p'_{y}$ crossing $H_{1}$, then
$\tau(e_{1}) \in (H_{1})^{+} \cap (H_{2})^{+}$.  If $e_{x}$ is the edge of $p'_{x}$ crossing $H_{1}$ then $\tau(e_{x}) \in (H_{1})^{+} \cap (H_{2})^{-}$.  It
follows from Lemma \ref{lemma:4pt} that $H_{1} \cap H_{2} \neq \emptyset$.  This proves the claim.  

We've now shown that there are geodesic edge-paths
$p'_{x}$, $p''_{y}$ connecting $z$ to $x$ and $y$ (respectively), and having the same initial edge $\hat{e}$.  We assume $z = \iota(\hat{e})$.  By the induction
hypothesis $[\tau(\hat{e}), y ] \cap [x, \tau(\hat{e})] \cap [x,y]$ is non-empty.  Since 
$$[\tau(\hat{e}), y ] \cap [x, \tau(\hat{e})] \cap [x,y]  \subseteq [z,y ] \cap [x,z] \cap [x,y]$$
by (3), the induction is complete.
\end{proof}

\bibliography{geofinal}
\nocite{*}
\bibliographystyle{plain}

\end{document}